\documentclass[10pt, draft]{amsart}
\usepackage{amssymb}
\usepackage{latexsym}
\usepackage{amsthm}
\usepackage{amsfonts}
\usepackage{amscd}
\usepackage{amsxtra}

\numberwithin{equation}{section}

\newtheorem{theorem}{Theorem}[section]
\newtheorem{lemma}[theorem]{Lemma}

\newtheorem{corollary}[theorem]{Corollary}

\theoremstyle{definition}
\newtheorem{definition}[theorem]{Definition}

\newtheorem{question}[theorem]{Question}

\newcommand{\N}{\mathbb{N}} 

\newcommand{\K}{\mathbb{K}} 

\newcommand{\U}{\mathcal{U}}

\makeatletter
\newcommand{\subalign}[1]{%
  \vcenter{%
    \Let@ \restore@math@cr \default@tag
    \baselineskip\fontdimen10 \scriptfont\tw@
    \advance\baselineskip\fontdimen12 \scriptfont\tw@
    \lineskip\thr@@\fontdimen8 \scriptfont\thr@@
    \lineskiplimit\lineskip
    \ialign{\hfil$\m@th\scriptstyle##$&$\m@th\scriptstyle{}##$\hfil\crcr
      #1\crcr
    }%
  }%
}
\makeatother

\title[d-Frequently hypercyclic operators which fail to be d-weakly mixing]{Existence of disjoint frequently hypercyclic operators which fail to be disjoint weakly mixing}

\author{\"Ozg\"ur Martin}

\address{Department of Mathematicst\\
 Mimar Sinan Fine Arts University\\
 Silah\c{s}\"or Cad. 71, Bomonti \c{S}i\c{s}li 34380, Istanbul, Turkey}

\email{ozgur.martin@msgsu.edu.tr}

\author{Yunied Puig}
\address{Department of Mathematics\\ 
University of California at Riverside\\
900 University Ave., Riverside, CA 92521, USA.}
\email{puigdedios@gmail.com}


\date{August 7, 2020}

\begin{document}

\begin{abstract}
	In this short note, we answer a question of Martin and Sanders [{Integr. Equ. Oper. Theory}, {85} (2) (2016), 191-220] by showing the existence of disjoint frequently hypercyclic operators which fail to be disjoint weakly mixing and, therefore, fail to satisfy the Disjoint Hypercyclicity Criterion. We also show that given an operator $T$ such that $T \oplus T$ is  frequently hypercyclic, the set of operators $S$ such that $T, S$ are disjoint frequently hypercyclic  but fail to satisfy the Disjoint Hypercyclicity Criterion is SOT dense in the algebra of bounded linear operators. 
\end{abstract}

\maketitle

\section{Introduction}

Let $\N$ denote the set of non-negative integers, $X$ be a separable and infinite dimensional Banach space over the  real or complex scalar field $\K$, and let $\mathcal{B}(X)$ denote the algebra of bounded linear operators on $X$. An operator $T \in \mathcal{B}(X)$ is called {\it hypercyclic} if there exists $x \in X$ such that $\{T^nx:n\in \N\}$ is dense in $X$ and such a vector $x$ is said to be a hypercyclic vector for $T$. By $HC(T)$, we will denote the set of hypercyclic vectors for $T$.

Disjointness in hypercyclicity is introduced independently by Bernal \cite{Ber07} and by B\`es and Peris \cite{BePe07} in 2007. For $N \geq 2$, operators $T_1, \ldots,T_N \in \mathcal{B}(X)$ are called {\it disjoint hypercyclic} or  {\it d-hypercyclic} if the direct sum operator $T_1 \oplus \dots \oplus T_N$ has a hypercyclic vector of the form $(x, \ldots, x) \in X^N$. Such a vector $x \in X$ is called a d-hypercyclic vector for $T_1, \ldots,T_N$. By $d$-$HC(T_1, \ldots,T_N)$, we will denote the set of d-hypercyclic vectors for $T_1, \ldots,T_N$. If $d$-$HC(T_1, \ldots,T_N)$ is dense, $T_1, \ldots,T_N$ are called to be {\it densely d-hypercyclic}. 

It is well known that for $T \in \mathcal{B}(X)$, the set $HC(T)$ is either empty or dense and it is non-empty if and only if $T$ is {\it topologically transitive}, that is for any two non-empty open $U,V \subset X$, there exists a positive integer $n$ such that $U \cap T^{-n}(V) \neq \emptyset$. Similarly, B\`es and Peris \cite{BePe07} proved that operators $T_1, \ldots,T_N$ are densely d-hypercyclic if and only if they are {\it d-topologically transitive}, that is for any non-empty open $U, V_1, \ldots, V_N \subset X$, there exists a positive integer $n$ such that $U \cap T_1^{-n}(V_1) \cap \ldots \cap T_N^{-n}(V_N) \neq \emptyset$. In \cite{SaSh14}, contrary to the single operator case, Sanders and Shkarin showed the existence of d-hypercyclic operators which are not densely d-hypercyclic and, therefore, fail to be d-topologically transitive.

We next remind a necessary condition for d-hypercyclicity, a natural extension of the Hypercyclicity Criterion which has played a significant role in linear dynamics. Note that for $N = 1$, the following definition gives the  single operator version of the Hypercyclicity Criterion.

\begin{definition} \cite{BePe07}\label{def:CoHC}
{\rm Let $(n_k)$ be a strictly increasing sequence of positive
integers. We say that $T_1,\dots,T_N\in \mathcal{B}(X)$ satisfy the {\it
d-Hypercyclicity Criterion with respect to $(n_k)$} provided there
exist dense subsets $X_0, X_1,\dots, X_N$ of $X$ and mappings
$S_{m, k}:X_m\to X$ with $1\le m \le N, \ k\in \N$ satisfying
\begin{equation} \label{eq:CoHC}
\begin{aligned}
 T_m^{n_k} &\underset{k\to\infty}{\longrightarrow} 0 \ \ \ \ \ \mbox{ pointwise on $X_0$,} \\
  S_{m, k} &\underset{k\to\infty}{\longrightarrow} 0 \ \ \ \ \ \mbox{ pointwise on $X_m$, and } \\
  \left( T_m^{n_k}S_{i,k}-\delta_{i,m}\, Id_{X_m} \right) &\underset{k\to\infty}{\longrightarrow } 0
  \ \ \ \ \ \mbox{ pointwise on $X_m$ $(1\le i\le N)$.}
\end{aligned}
\end{equation}
In general, we say that $T_1,\dots,T_N$ satisfy the d-Hypercyclicity
Criterion if there exists some sequence $(n_k)$ for which
\eqref{eq:CoHC} is satisfied.}
\end{definition}

\begin{theorem} \cite[Theorem 2.7]{BePe07} \label{T:HDCoHC}
$T_1,\dots, T_N \in \mathcal{B}(X)$  satisfy the d-Hypercyclicity
Criterion if and only if for each $r\in\N$, the direct sum operators $\bigoplus_{i=1}^{r}T_1, \dots , \bigoplus_{i=1}^{r}T_N$ are d-topologically transitive on $X^r$.
\end{theorem}

In \cite{BePe99}, B\`es and Peris showed that an operator $T \in \mathcal{B}(X)$ satisfies the Hypercyclicity Criterion if and only if it is {\it weakly mixing}, that is $T \oplus T$ is topologically transitive. An older result of Furstenberg \cite{Fu67} also shows that $T$ is weakly mixing if and only if for each $r\in\N$, the direct sum operator $\bigoplus_{i=1}^{r}T$ is topologically transitive. In a landmark result, De la Rosa and Read \cite{DeRe09} constructed a Banach space that supports a hypercyclic operator which is not weakly mixing, and thus fails to satisfy the Hypercyclicity Criterion. 

In the disjointness case, the picture is again different. We say $T_1,\dots, T_N \in \mathcal{B}(X)$ are {\it d-weakly mixing} if $T_1 \oplus T_1, \ldots, T_N \oplus T_N$ are d-topologically transitive on $X^2$. In \cite{SaSh14}, Sanders and Shkarin also showed that every Banach space supports d-weakly mixing operators which fail to satisfy the d-Hypercyclicity Criterion. Then combining the above-mentioned results, we have the following implications:\\

\begin{center}
{\it d-Hypercyclicity Criteron $\Rightarrow$ d-weakly mixing $\Rightarrow$ d-topologically transitive $\Rightarrow$ d-hypercyclic}
\end{center}

\vspace{2mm}

\noindent and none of the reverse implications hold.

In this short note, we are interested in the disjoint version of a strong recurrence property of hypercyclic operators, called frequent hypercyclicity, which is introduced by Bayart and Grivaux \cite{BaGr06}. An operator $T \in \mathcal{B}(X)$ is called {\it frequently hypercyclic} if there exists some $x \in X$ such that for every non-empty open set $U \subset X$ the set  $\{n : T^nx \in U\} \subset \N$ has positive lower density. Such a vector $x$ is called a {\it frequently hypercyclic vector} for the operator $T$. Recall that the lower density of a set $A \subset \N$ is defined by
\begin{equation*}
	\mbox{\underline{dens}} A:=\liminf_{N\to \infty}\frac{card\{n \leq N: n\in A\}}{N}.
\end{equation*}
Frequent hypercyclicity cleary implies hypercyclicity. In fact, Grosse-Erdmann and Peris \cite{GrPe05} showed that frequently hypercyclic operators are weakly mixing and, therefore, they satisfy the Hypercyclicity Criterion. 

We say $T_1, \ldots , T_N \in \mathcal{B}(X)$ with $N \geq 2$ are {\it d-frequently hypercyclic} if there exists a vector $x$ in $X$ such that the vector $(x, \dots, x)$ is a frequently hypercyclic vector for the direct sum operator $T_1 \oplus  \cdots \oplus T_N$ on $X^N$. Clearly, d-frequent hypercyclicity  implies d-hypercyclicity, however it is not clear whether d-frequent hypercyclicity implies  d-weakly mixing or even d-topological transitivity (dense d-hypercyclicity). In the next section, we will answer in the negative the following questions posed in \cite{MaSa16}:\\

\begin{question}
  If $T_1,T_2 \in \mathcal{B}(X)$ are d-frequently hypercyclic, must they be d-weakly mixing? Must they satisfy the d-Hypercyclicity Criterion?\\
\end{question}

\section{d-Frequently hypercyclic operators which fail to be d-weakly mixing}

For any $T, T_1, \ldots, T_N \in \mathcal{B}(X)$, let $FHC(T)$ and $d$-$FHC(T_1, \ldots, T_N)$ denote the sets of frequently hypercyclic vectors of $T$ and d-frequently hypercyclic vectors of $T_1, \ldots, T_N$, respectively. Note that, if $(f_1, \ldots, f_N) \in HC(T \oplus \ldots \oplus T)$ then the vectors $f_1, \ldots, f_N$ must be linearly independent. By modifying the results of Sanders and Shkarin in \cite{SaSh14}, we first show the existence of d-frequently hypercyclic operators which fail to be d-weakly mixing.

\begin{lemma}\label{lemma_invertible}
	Let $T,L$ be in $\mathcal{B}(X)$ and $L$ be invertible. If $S := L^{-1}TL$, then $f \in d$-$FHC(T, S)$ if and only if $(f, Lf) \in FHC(T \oplus T)$.
\end{lemma}
\begin{proof}
If $f \in d$-$FHC(T, S)$,  then for any two non-empty, open sets $U,V \subset X$  we have
\[
	\mbox{\underline{dens}}\{n \in \N: T^n f \in U, S^n f \in L^{-1}(V) \}  > 0.
\]
This gives that $\mbox{\underline{dens}}\{n \in \N: T^n f \in U, T^n L f \in V \} > 0$ and, therefore, $(f, Lf) \in FHC(T \oplus T)$.

Now if we assume $(f, Lf) \in FHC(T \oplus T)$ and  $U,V \subset X$ are non-empty and open, then by invertibility of $L$, we can say that 
\[
	\mbox{\underline{dens}}\{n \in \N: T^n f \in U, T^n Lf \in L(V) \}  > 0.
\]
But, this implies $\mbox{\underline{dens}}\{n \in \N: T^n f \in U, L^{-1}T^n L f \in V \}  > 0$ and, thus, $f \in d$-$FHC(T, S)$.
\end{proof}
\begin{lemma}\label{lemma_iterations}
	Let $T$ be in $\mathcal{B}(X)$ and $(f,g) \in FHC(T \oplus T)$. 
\begin{enumerate}
	\item For any $r \in \N$, $(T^r f, g) \in FHC(T \oplus T)$.
	\item For any non-zero $c \in \K$,  $(f, f + cg) \in FHC(T \oplus T)$.
\end{enumerate}
\end{lemma}
\begin{proof}
To prove the first statement, let $U,V \subset X$ be any two non-empty open sets and $r \in \N$. We have that
\[
	\mbox{\underline{dens}}\{n \in \N: T^n f \in T^{-r}(U), T^n g \in V \}  > 0,
\]
which implies $\mbox{\underline{dens}}\{n \in \N: T^n(T^rf) \in U, T^n g \in V \}  > 0$, and $(T^r f, g) \in FHC(T \oplus T)$.

For the second statement, let $c \in \K$ with $c \neq 0$ and  $U,V \subset X$  be non-empty and open sets. Choose $u_0 \in U$ and $v_0 \in V$ and define $x_0 := v_0 - u_0$. Choose $\varepsilon > 0$ such that $B(u_0, \varepsilon) \subset U$ and $B(v_0, \varepsilon) \subset V$ where $B(x, r)$ denotes the ball centered at $x \in X$ with radius $r > 0$. 

By our assumption, we have that 
\[
	\mbox{\underline{dens}}\{n \in \N: T^n f \in B(u_0, \varepsilon / 2), T^n g \in c^{-1} B(x_0, \varepsilon / 2) \}  > 0.
\]
This implies $\mbox{\underline{dens}}\{n \in \N: T^n f \in B(u_0, \varepsilon / 2), T^n f + cT^n g \in B(v_0, \varepsilon) \}  > 0$, and therefore,
\[
	\mbox{\underline{dens}}\{n \in \N: T^n f \in U, T^n (f + cg) \in V \}  > 0.
\]
Thus, $(f, f + cg) \in FHC(T \oplus T)$.
\end{proof}

\begin{theorem}\label{theorem_weakly_mixing}
\label{theorem_izmir}
Let $X$ be a separable  Banach space and $T \in \mathcal{B}(X)$ such that $T \oplus T$ is frequently hypercyclic on $X \times X$. Then, there exists  $S \in \mathcal{B}(X)$ such that $T, S$ are densely d-frequently hypercyclic  but they fail to be d-weakly mixing.
\end{theorem}

\begin{proof}
Let $(f,g) \in FHC(T \oplus T)$ where $f$ and $g$ are linearly independent. By the Hahn-Banach Theorem, there exists a linear functional $\lambda$ in the dual space $X^*$ such that $\lambda(f) = 1$ and $\lambda(g) = 0$. Define $L \in \mathcal{B}(X)$ by $Lx = x +\lambda(x)g$ for $x \in X$. Then $L$ is an invertible operator with the inverse $L^{-1}x = x - \lambda(x)g$. Now, define $S := L^{-1}TL$. 

First, we show that $T, S$ are densely d-frequently hypercyclic. To this end, let $U \subset X$ be non-empty and open and choose $r \in \N$ such that $T^r f \in U$ and $\lambda(T^r f) \neq 0$.

Since $(f,g) \in FHC(T \oplus T)$, we have that $(T^r f,g) \in FHC(T \oplus T)$ by Lemma \ref{lemma_iterations}(1) and  $(T^r f, T^r f + \lambda(T^r f )g) \in FHC(T \oplus T)$ by Lemma \ref{lemma_iterations}(2). The last expression means $(T^r f, LT^r f)  \in FHC(T \oplus T)$ which in turn gives that $T^r f \in d$-$FHC(T, S)$ by Lemma \ref{lemma_invertible}.

Now, in order to reach a contradiction, assume that $T \oplus T$, $S \oplus S$ are d-hypercyclic. Then there exists a $(x,y) \in X \times X$ such that the set 
$$
	\{(T^n x, T^n y, S^n x, S^n y): n \geq 0\} = \{(T^n x, T^n y, L^{-1}T^n Lx,  L^{-1}T^n Ly): n \geq 0\}
$$ 
is dense in $X^4$.  Therefore,  $(x, y, Lx, Ly) \in HC(T \oplus T \oplus T \oplus T)$.

As in  Lemma \ref{lemma_iterations}, from the last statement we can derive 
\[
	(x, y, Lx - x, Ly - y) \in HC(T \oplus T \oplus T \oplus T).
\]
This implies $(x, y, \lambda(x)g, \lambda(y)g) \in HC(T \oplus T \oplus T \oplus T)$, or in particular, $(\lambda(x)g, \lambda(y)g) \in HC(T \oplus T)$. But, this is a contradiction since the vectors $\lambda(x)g$ and $\lambda(y)g$ are linearly dependent.
\end{proof}

In \cite{DeFrGrPe12}, De la Rosa et al. showed that any complex separable infinite-dimensional Banach space with an unconditional Schauder decomposition supports an operator $T$ such that $T \oplus T$ is frequently hypercyclic. Indeed, $T$ has a perfectly spanning set of eigenvectors associated to unimodular eigenvalues, so does $T\oplus T$. Hence, we have the following.

\begin{corollary} \label{corollary_existence}
	Every complex separable infinite-dimensional Banach space with an unconditional Schauder decomposition supports a densely d-frequently hypercyclic tuple of operators which fail to be d-weakly mixing and, therefore, fail to satisfy the d-Hypercyclicity Criterion.
\end{corollary}

Chan \cite{Ch01} showed that the hypercyclic operators on a separable infinite dimensional Hilbert space form a dense subset of the algebra of continuous linear operators in the strong operator topology and, later, this result is extended to Frechet spaces by B\`es and Chan  \cite{BeCh03}. In \cite{MaSa20}, the first author and Sanders showed that for any d-hypercyclic $T_1, \ldots, T_N \in \mathcal{B}(X)$ the set of operators $S \in \mathcal{B}(X)$ such that $T_1, \ldots, T_N, S$ remain to be d-hypercyclic is SOT-dense in $X$. Motivated by these results, we give the next theorem.

\begin{theorem}
\label{theorem_lima}
	Let $X$ be a separable  Banach space and $T \in \mathcal{B}(X)$ such that $T \oplus T$ is frequently hypercyclic 
	on $X \times X$. Then the set 
	$$
		\{S \in \mathcal{B}(X): T, S \text{ are d-frequently hypercyclic but do not satisfy the d-Hypercyclicity Criterion}\}
	$$
	is SOT-dense in $\mathcal{B}(X)$.
\end{theorem}

\begin{proof}
	Let $T \in \mathcal{B}(X)$ so that $T \oplus T$ is frequently hypercyclic on $X \times X$. Let $A \in \mathcal{B}(X)$ be arbitrary and 
	$$\U_{e_1,\ldots,e_N, \epsilon} := \{B \in \mathcal{B}(X): \|B(e_i) - A(e_i)\| < \epsilon, 1 \leq i \leq N\}$$
be a SOT-neighborhood of $A$ where $e_1,\ldots,e_N \in X$ are linearly independent and $\epsilon >0$.

Since $X$ supports a frequently hypercyclic operator, $X$ is infinite dimensional and one can find $f_1, \ldots, f_N \in X$ such that $\|f_i - A(e_i)\| < \epsilon$ for $1 \leq i \leq N$, and $e_1,\ldots, e_N, f_1, \ldots, f_N$ are linearly independent.

Now choose $x_1, \ldots, x_N, Tx_1, \ldots, Tx_N \in FHC(T)$ and $(f,g) \in FHC(T \oplus T)$ so that  the set
$$
	\mathcal{I} := \{e_1,\ldots, e_N, f_1, \ldots, f_N, x_1, \ldots, x_N, Tx_1, \ldots, Tx_N, f, g\}
$$
is linearly independent.

For each $1 \leq i \leq N$, pick $x_i^*, y_i^* \in X^*$ such that $x_i^*(e_i) = x_i^*(x_i) = 1$ with $x_i^* \equiv 0$  on $\mathcal{I} \backslash \{e_i, x_i\}$ and $y_i^*(f_i) = y_i^*(Tx_i) = 1$ with $y_i^* \equiv 0$  on $\mathcal{I} \backslash \{f_i, Tx_i\}$. Let $E := span\{e_1,\ldots, e_N, f_1, \ldots, f_N\}$, $F := span \{x_1, \ldots, x_N, Tx_1, \ldots, Tx_N\}$, and $Z := \bigcap_{i=1}^N (\ker x_i^* \cap \ker y_i^*)$. Then, $X = E \oplus Z = F \oplus Z$.

Now, note that $(f,g) \in FHC(T \oplus T) \cap (Z \times Z)$. By the Hahn-Banach theorem, we can choose a $\lambda \in X^*$ such that $\lambda(f) = 1$ and $\lambda(g) = 0$ and define $L \in \mathcal{B}(X)$ as follows:

For any $x \in X = E \oplus Z$, there exist unique $y \in E$ and $z \in Z$ such that $x = y+ z$ with $y$ in the form 
$$
	y = \sum_{j=1}^N \alpha_j e_j + \sum_{j=1}^N \beta_j f_j. 
$$
Then define $L(x)$  as
$$
	L(x) = \sum_{j=1}^N \alpha_j x_j + \sum_{j=1}^N \beta_j Tx_j + z  + \lambda(z)g.
$$
It is easy to see that $L$ is invertible where for any $x \in X = F \oplus Z$ with $x = y + z$, $y =  \sum_{j=1}^N \alpha_j x_j + \sum_{j=1}^N \beta_j Tx_j  \in F$, and $z \in Z$, we have 
$$
	L^{-1}(x) =  \sum_{j=1}^N \alpha_j e_j + \sum_{j=1}^N \beta_j f_j + z - \lambda(z)g.
$$

Now, if we define $S := L^{-1}TL$, then $S \in \U_{e_1,\ldots,e_N, \epsilon}$ since $Se_i = f_i$ for $1 \leq i \leq N$. By Lemma \ref{lemma_iterations}, $(f,g) \in FHC(T \oplus T)$ implies that  $(f, f + \lambda(f)g) \in FHC(T \oplus T)$. Since we also have that $f \in Z$, $Lf = f + \lambda(f)g$, thus $(f, Lf) \in FHC(T \oplus T)$  and, as before, $f \in d$-$FHC(T, S)$ by Lemma \ref{lemma_invertible}.

Lastly, we need to show that $T, S$ do not satisfy the Disjoint Hypercyclicity Criterion. By Theorem \ref{T:HDCoHC}, it is enough to show that $\bigoplus_{i=1}^{4N+2}T, \bigoplus_{i=1}^{4N+2}S$ cannot be disjoint topologically transitive. By way of contradiction, assume $(u_1, \ldots, u_{4N+2})$ be a disjoint hypercyclic vector for the direct sums $\bigoplus_{i=1}^{4N+2}T, \bigoplus_{i=1}^{4N+2}S$. This means that 
$$
	\{(T^nu_1, \ldots, T^nu_{4N+2}, S^nu_1, \ldots, S^nu_{4N+2}): n \geq 0\}
$$ 
is dense in $X^{8N+4}$. This, in turn, implies that 
$$
	(u_1, \ldots, u_{4N+2}, Lu_1, \ldots, Lu_{4N+2}) \in HC\left(\bigoplus_{i=1}^{8N+4}T\right).
$$
By Lemma \ref{lemma_iterations}, we conclude that 
$$
	(u_1, \ldots, u_{4N+2}, Lu_1 - u_1, \ldots, Lu_{4N+2} - u_{4N+2}) \in HC\left(\bigoplus_{i=1}^{8N+4}T\right),
$$
or 
\begin{equation}\label{eq:Lx-x}
	(Lu_1 - u_1, \ldots, Lu_{4N+2} - u_{4N+2}) \in HC\left(\bigoplus_{i=1}^{4N+2}T\right).
\end{equation}

Now, it is enough to observe that, for $1 \leq i \leq 4N+2$, we have
$$
	Lu_i - u_i \in span\{e_1,\ldots, e_N, f_1, \ldots, f_N, x_1, \ldots, x_N, Tx_1, \ldots, Tx_N, g\},
$$
and, therefore, the set $\{Lu_i - u_i : 1 \leq i \leq 4N+2\}$ is linearly dependent, contradicting (\ref{eq:Lx-x}).
\end{proof}

We remind the reader that determining whether $T\oplus T$ is frequently hypercyclic whenever $T$ is frequently hypercyclic is still an open problem since Bayart and Grivaux \cite{BaGr06} posed it for the first time in 2006. Therefore, we cannot remove the condition of frequent hypercyclicity of $T\oplus T$ in Theorem \ref{theorem_izmir} and Theorem \ref{theorem_lima}. However, we have a different panorama concerning $\mathcal{U}$-frequent hypercyclicty and d-reiterative hypercyclicity. We get these notions when we replace the positive lower density condition in the definition of frequent hypercyclicty by positive upper density and positive upper Banach density, respectively (see \cite{BeMePePu16}). A recent result by Ernst et al. \cite{ErEsMe}, asserts that $T$ is $\mathcal{U}$-frequently hypercyclic (reiteratively hypercyclic) if and only if $T\oplus T$ is  $\mathcal{U}$-frequently hypercyclic (reiteratively hypercyclic). Now, observe that the proofs of Theorem \ref{theorem_izmir} and Theorem \ref{theorem_lima} adapt easily to the notions of $\mathcal{U}$-frequent hypercyclicity and reiterative hypercyclicity without any conditions on $T \oplus T$. To the best of our knowledge, the following question is open:
 
 \begin{question}
	Does every separable Banach space support a reiteratively hypercyclic operator?
\end{question}

In \cite{BeMaPeSh12}, B\`es et al. showed the existence of a mixing operator $T$ such that $T, T^2$ are not d-mixing and asked whether there exists a mixing operator $T$ such that $T, T^2$ are not even d-topologically transitive. Using a result in ergodic Ramsey theory, the second author \cite{Pu17} answered this question in the positive by showing that the same operator $T$ given by B\`es et al. is also chaotic, and  $T, T^2$ fail to be d-topologically transitive. So, we pose the following question:

\begin{question}  
	Is there an example of a frequently hypercyclic operator $T$ for which $T,T^2$ are not  d-topologically transitive?
\end{question}
 
We end this note by restating the following open question which was posed also in \cite{MaSa16}:

\begin{question}  
	If $T_1,T_2 \in \mathcal{B}(X)$ are d-frequently hypercyclic, must they be densely d-hypercyclic (equivalently, d-topologically transitive)?
\end{question}

\end{document}